\documentclass[a4paper]{scrartcl}

\usepackage[T1]{fontenc}
\usepackage{amsmath}
\usepackage{amssymb,bbm}
\usepackage{amsthm}
\usepackage{enumerate}

\theoremstyle{plain}
\newtheorem{theorem}{Theorem}[section]
\newtheorem{lemma}[theorem]{Lemma}
\newtheorem{proposition}[theorem]{Proposition}

\theoremstyle{definition}
\newtheorem{definition}[theorem]{Definition}%[theorem]
\newtheorem{example}[theorem]{Example}%[theorem]
\newtheorem{remark}[theorem]{Remark}%[theorem]

\DeclareMathOperator{\ob}{ob}
\DeclareMathOperator{\mor}{mor}

\DeclareMathOperator{\Aut}{Aut}

\DeclareMathOperator{\id}{id}

\DeclareMathOperator{\Supp}{Supp}

\DeclareMathOperator{\Ring}{Ring}

\DeclareMathOperator{\ideal}{ideal}

\begin{document}

\title{Simple Rings and Degree Maps}

\author{Patrik Nystedt\footnote{Address: University West, Department of Engineering Science, SE-46186 Trollh\"{a}ttan, Sweden, E-mail: Patrik.Nystedt@hv.se} \and Johan \"{O}inert\footnote{Address: Centre for Mathematical Sciences, Lund University, P.O. Box 118, SE-22100 Lund, Sweden, E-mail: Johan.Oinert@math.lth.se}}

\maketitle

\begin{abstract}
For an extension $A/B$ of neither necessarily associative
nor necessarily unital rings,
we investigate the connection
between simplicity of $A$ with
a property that we call
$A$-simplicity of $B$.
By this we mean that
there is no non-trivial ideal $I$
of $B$ being $A$-invariant, that is satisfying $AI \subseteq IA$. 
We show that $A$-simplicity of $B$ is
a necessary condition for simplicity of $A$
for a large class of ring extensions
when $B$ is a direct summand of $A$.
To obtain sufficient conditions for 
simplicity of $A$, we introduce the concept
of a degree map 
for $A/B$.
By this we mean a map $d$ from $A$ to the
set of non-negative integers 
satisfying the following two conditions
(d1) if $a \in A$, then 
$d(a) = 0$ if and only if $a=0$;
(d2) there is a subset $X$ of $B$ generating
$B$ as a ring such that for each non-zero ideal $I$
of $A$ and each non-zero $a \in I$
there is a non-zero $a' \in I$ with $d(a') \leq d(a)$
and $d(a'b - ba') < d(a)$ for all $b \in X$.
We show that if the centralizer $C$ 
of $B$ in $A$ is an $A$-simple ring,
every intersection of $C$
with an ideal of $A$ is $A$-invariant,
$A C A = A$ and there is a degree map for $A/B$,
then $A$ is simple.
We apply these results to various types of graded 
and filtered rings, such as skew group rings,
Ore extensions and Cayley-Dickson doublings.
\end{abstract}

%\begin{keyword}
%simplicity, degree map, ring extension, ideal associativity
%\MSC[2010] 12E15 \sep 16D25 \sep 16S32 \sep 16S35 \sep 16S36 \sep 16W50 \sep 17A99
%\end{keyword}

\section{Introduction}\label{introduction}

Let $A/B$ be a ring extension.
By this we mean that $A$ and $B$ are rings that are neither necessarily associative
nor necessarily unital with $B$ contained in $A$.
For general ring extensions, simplicity of $A$
is, of course, neither a necessary nor 
a sufficient condition for simplicity of $B$.
However, in many cases where $A$ is graded or filtered,
and $B$ sits in $A$ as a direct summand,
simplicity of $A$ is connected to weaker simplicity
conditions for $B$.
In particular, this often holds if there is an action on 
the ring $B$ induced by the ring structure on $A$. 
The aim of this article is to investigate
both necessary (see Theorem \ref{firsttheorem}) and
sufficient (see Theorem \ref{secondtheorem}) conditions
on $B$ for simplicity of $A$.
The impetus for our approach is threefold.

The first part of our motivation comes
from group graded ring theory and topological dynamics.
A lot of attention has been given the 
connection between properties of topological
dynamical systems $(X,s)$ and the ideal structure
of the corresponding $C^*$-algebras $C^*(X,s)$
(see e.g. \cite{davidson96}, \cite{power78},
\cite{tomiyama87}, \cite{tomiyama92} and 
\cite{williams07}).
In particular, Power \cite{power78} has shown
that if $X$ is a compact Hausdorff space of infinite cardinality,
then $C^*(X,s)$ is simple if and only if 
$(X,s)$ is minimal.
Inspired by this result,
Öinert \cite{oinert09,oinertarxiv11} 
has shown that there is an analogous result
for skew group algebras defined by
general group topological dynamical systems.
In fact Öinert loc. cit.
shows this by first establishing the following result
for general skew group rings
(for more details concerning topological 
dynamics, see Section \ref{applicationtopological}). 

\begin{theorem}[Öinert \cite{oinert09,oinertarxiv11}]\label{oinertAabelian}
Suppose that $B$ is an associative, unital ring
and $A= B \rtimes^{\sigma} G$ is a skew group ring.
{\rm (a)} If $B$ is commutative,
then $A$ is simple if and only if $B$ is $G$-simple
and $B$ is a maximal commutative subring of $A$;
{\rm (b)} If $G$ is abelian, then $A$ 
is simple if and only if $B$ is $G$-simple
and $Z(A)$ is a field.
\end{theorem}

The second part of our motivation comes
from the filtered class of rings called Ore extensions
$A = B[x ; \sigma , \delta]$.
A lot of work has been devoted to the
study of the connection between the ideal structure of
$A$ and $B$ in this situation
(see e.g. \cite{cozzens75}, \cite{goodearl82}, 
\cite{jacobson37}, \cite{jordan75} and \cite{oinricsil11}). 
In \cite{oinricsil11} Öinert, Richter and Silvestrov
show that there are simplicity results
for differential polynomial rings 
that are almost completely analogous
to the skew group ring situation
(for more details, see Section \ref{oreextensions}).

\begin{theorem}[Öinert, Richter and Silvestrov \cite{oinricsil11}]\label{oinricsilmaximalcommutative}
Suppose that $B$ is an associative and unital ring 
and $A=B[x;\delta]$ is a differential polynomial ring.
{\rm (a)} If $B$ is a $\delta$-simple and
maximal commutative subring of the differential polynomial ring $A$,
then $A$ is simple;
{\rm (b)} The ring $A$ is simple if and only if
$B$ is $\delta$-simple and
$Z(A)$ is a field.
\end{theorem}

The third part of our motivation 
comes from non-associative ring theory.
Namely, starting with the real numbers, using
the classical Cayley-Dickson doubling procedure
(see Section \ref{applicationcayley}),
we get successively, 
the complex numbers,
Hamilton's quaternions,
the octonions, 
the sedenions and so on. 
Although none of these rings, from
the octonions on, are associative,
it is well known that they are all simple.
Furthermore, matrix algebras over such rings
are also simple but, of course, in general not associative.

In this article, we unify all of the three types
of simplicity results for ring extensions $A/B$
mentioned above, by introducing the notion of {\it $A$-simplicity} of $B$
(see Definition \ref{defAinvariant}). 
It turns out that for crossed product algebras
and Ore extensions, $A$-simplicity coincides
with respectively $G$-simplicity 
(see Proposition \ref{crossedproductGinvariant})
and $\sigma$-$\delta$-simplicity
(see Proposition \ref{oreidealABideal}).
We also introduce a weak form of
associativity that we call
{\it ideal associativity} 
(see Definition \ref{defidealassociative}).
We show that $A$-simplicity of $B$
is often a necessary condition
for simplicity of $A$. 

\begin{theorem}\label{firsttheorem}
If $A/B$ is an ideal associative ring extension satisfying
{\rm (i)} $B$ is a direct summand 
of $A$ as a left (right) $B$-module, 
{\rm (ii)} every ideal of $B$ has the identity property
as a right (left) $B$-module and 
{\rm (iii)} $A$ is simple, 
then $B$ is $A$-simple.
\end{theorem}

Many of the proofs of the sufficient conditions for 
simplicity of $A$ are often based on some 
kind of reduction argument on the number
of homogeneous components that are present
in elements of ideals.
We show that many of these arguments can be treated
using a notion that we call a {\it degree map}
(see Definition \ref{defdegreemap})
for the extension $A/B$.

\begin{theorem}\label{secondtheorem}
If $A/B$ is a ring extension satisfying
{\rm (i)} $C_A(B)$ is a ring which is $A$-simple,
{\rm (ii)} every intersection of $C_A(B)$
with an ideal of $A$ is $A$-invariant,
{\rm (iii)} $A C_A(B) A = A$ and 
{\rm (iv)} there is a degree map for $A/B$,
then $A$ is simple.
\end{theorem}

In Section \ref{mainresults},
we define the relevant notions concerning
ring extensions that will be used throughout this article
and we show Theorem \ref{firsttheorem} and Theorem \ref{secondtheorem}.
In Section \ref{categorygradedrings},
we apply Theorem \ref{firsttheorem}
and Theorem \ref{secondtheorem} to
category graded rings.
In particular, we show that there are
versions of Theorem \ref{oinertAabelian}
that hold for 
strongly groupoid graded rings. 
In Section \ref{crossedproductalgebras},
we utilize results from Section \ref{categorygradedrings}
in order to analyze simplicity of
category crossed products.
In Sections \ref{applicationcayley}--\ref{oreextensions}, 
we apply the results from the previous sections to
obtain simplicity results for, 
respectively, Cayley extensions, twisted and skew matrix algebras
over non-associative rings,
skew group rings over
non-associative algebras associated with topological dynamical systems, and Ore extensions.

\section{Preliminaries and Proofs of the Main Results}\label{mainresults}

In this section, we fix the relevant notions 
concerning ring extensions that we will
use in the sequel and we also prove Theorem
\ref{firsttheorem} and Theorem \ref{secondtheorem}. 
To this end, we show  
three results (see Proposition \ref{idealassociative},
Proposition \ref{firstprop} and Lemma \ref{thirdtheorem}) concerning
ideals of ring extensions.

Let $A/B$ be a ring extension and suppose that
$X$, $Y$ and $Z$ are subsets of $A$.
We let $XY$ denote the set of all finite sums
of elements of the form $xy$, for $x \in X$
and $y \in Y$.
Furthermore, we let $XYZ$ denote the 
set of all finite sums of elements 
of the form $x(yz) + (x'y')z'$
for $x,x' \in X$, $y,y' \in Y$ and $z,z' \in Z$.
If $A$ is unital, then the multiplicative
identity element of $A$ is denoted by $1_A$.
Recall that 
the \emph{centralizer}
of $B$ in $A$, denoted by $C_A(B)$, is the
set of elements in $A$ that commute
with every element in $B$.
If $C_A(B)=B$, then $B$ is said to be a {\it maximal commutative} subring of $A$.
The set $C_A(A)$ of $A$ is called the
{\it center} of $A$ and is denoted by $Z(A)$.
We say that a left (or right) $B$-module $X$
has the {\it identity property} if
$BX=X$ (or $XB=X$).
By an \emph{ideal} we always mean a two-sided ideal.
The pre-ordered set of ideals of $A$
is denoted by $\ideal(A)$.

\begin{definition}\label{defAinvariant}
We say that an ideal $I$ of $B$
is {\it $A$-invariant} if $AI \subseteq IA$.
We say that $A/B$ is $A$-{\it simple} if  
there is no non-trivial
$A$-invariant ideal $I$ of $B$. 
We let the collection of $A$-invariant ideals
of $B$ be denoted by $\ideal_A(B)$.
\end{definition}

\begin{definition}\label{defidealassociative}
We say that $A/B$ is {\it ideal associative} 
if any ideal $I$ of $B$
associates with any finite collection of copies of $A$.
So for instance using two copies of $A$
we get that $(IA)A = I(AA)$, $(AI)A = A(IA)$
and $(AA)I = A(AI)$. 
Clearly, if $A$ is associative
then $A/B$ is necessarily ideal associative.
\end{definition}

\begin{definition}\label{defdegreemap}
We say that a function $d$ from $A$ to the
set of non-negative integers is a
{\it degree map} for $A/B$ if it satisfies the following two conditions:
\begin{itemize}
	\item[{\rm (d1)}] if $a \in A$, then 
$d(a) = 0$ if and only if $a=0$;
	\item[{\rm (d2)}] there is a subset $X$ of $B$ generating
$B$ as a ring, such that for each non-zero ideal $I$
of $A$ and each non-zero $a \in I$
there is a non-zero $a' \in I$ satisfying $d(a') \leq d(a)$
and $d(a'b - ba') < d(a)$ for all $b \in X$.
\end{itemize}
\end{definition}

\begin{proposition}\label{idealassociative}
If $A/B$ is an ideal associative ring extension, 
then there is a map of pre-ordered sets
$i : \ideal_A(B) \rightarrow \ideal(A)$,
defined by $i(I)=IA$, 
for all $I \in \ideal_A(B)$.
\end{proposition}

\begin{proof}
All that we need to show is that $i$ is well defined.
Take an $A$-invariant ideal $I$ of $B$.
Then $A(IA) = (AI)A \subseteq (IA)A = I(AA)
\subseteq IA$ and
$(IA)A = I(AA) \subseteq IA$. 
Therefore, $IA$ is an ideal of $A$.
\end{proof}

\begin{proposition}\label{firstprop}
If $A/B$ is an ideal associative 
ring extension satisfying
\begin{itemize}
\item[{\rm (i)}] $B$ is a direct summand 
of $A$ as a left (right) $B$-module and 

\item[{\rm (ii)}] every ideal of $B$ has the identity property
as a right (left) $B$-module, 

\end{itemize}
then $i : \ideal_A(B) \rightarrow \ideal(A)$ is
an injective map of pre-ordered sets.
\end{proposition}

\begin{proof}
We show the left-right part of the claim.
The right-left part of the claim is shown
in a completely analogous way and is therefore
left to the reader.
Let the set of additive subgroups of $A$
be denoted by $S(A)$.
Let $p$ denote the map
$S(A) \ni Y \mapsto B \cap Y \in S(B)$.
We will show that $p \circ i = \id_{\ideal_A(B)}$,
which, in particular, implies injectivity of $i$.
To this end, take an $A$-invariant ideal $I$ of $B$
and a left $B$-bimodule $X$ such that
$A = B \oplus X$ as left $B$-modules.
Since $I$ has the identity property
as a right $B$-module we get that
$(p \circ i)(I) = B \cap IA =
B \cap \left( IB \oplus IX \right) =
B \cap \left( I \oplus IX \right) = I$.
\end{proof}

\noindent {\bf Proof of Theorem \ref{firsttheorem}.}
This follows 
from Proposition \ref{firstprop}.
\hfill $\qed$

\begin{definition}[\cite{oinlun10}]
$A/B$ is said to have the
{\it ideal intersection property}
if every non-zero ideal of $A$ has 
non-zero intersection with $B$.
\end{definition}

\begin{lemma}\label{thirdtheorem}
If $A/B$ is a ring extension
equipped with a degree map,
then $A / C_A(B)$ has the ideal
intersection property.
\end{lemma}

\begin{proof}
Let $I$ be a non-zero ideal of $A$.
Take $a \in I$ with 
$d(a) = {\rm min}[ d(I \setminus \{ 0 \}) ]$.
By (d2) there is a non-zero $a' \in I$ 
with $d(a') \leq d(a)$ and
$d(a'b-ba') < d(a)$ for all $b \in X$. 
Since $a'b-ba' \in I$ it follows from the assumptions
on $d(a)$ that $a'b = ba'$ for all $b \in X$.
Since $X$ generates $B$ as a ring, we can 
conclude that $a' \in C_A(B)$.
\end{proof}

\noindent {\bf Proof of Theorem \ref{secondtheorem}.}
Suppose that $A/B$ is a ring extension
satisfying (i)-(iv)
in the formulation of Theorem \ref{secondtheorem}.
Let $I$ be a non-zero ideal of $A$.
From Lemma \ref{thirdtheorem} it follows that
$I \cap C_A(B)$ is a non-zero
$A$-invariant ideal of $C_A(B)$, which,
from $A$-simplicity of $C_A(B)$,
implies that $I \cap C_A(B) = C_A(B)$,
i.e. that $C_A(B) \subseteq I$.
Therefore $A = A C_A(B) A \subseteq 
A I A \subseteq I$ and hence $A = I$.
This shows that $A$ is simple. 
\hfill $\qed$

\section{Applications}

In the Sections \ref{categorygradedrings}-\ref{oreextensions} below,
we demonstrate how the notion of a degree map
can be useful in the study of various classes of ring extensions.

\subsection{Category Graded Rings}\label{categorygradedrings}

In this section, we apply Theorem
\ref{firsttheorem} and Theorem \ref{secondtheorem} 
to category graded rings.
Thereby, we obtain both necessary
(see Proposition \ref{Rsimpleimplies}) and sufficient
(see Proposition \ref{propgroupoidsimple})
conditions for such rings to be simple.

\begin{definition}
Let $G$ be a category.
The family of objects of $G$ is denoted by $\ob(G)$.
We will often identify an object of $G$ with
its associated identity morphism.
The family of morphisms in $G$ is denoted by $\mor(G)$.
Throughout this article $G$ is assumed to be small, that is
with the property that $\mor(G)$ is a set.
The domain and codomain of a morphism $g$ in $G$ is denoted by
$d(g)$ and $c(g)$ respectively.
We let $G^{(2)}$ denote the collection of composable
pairs of morphisms in $G$, that is all $(g,h)$ in
$\mor(G) \times \mor(G)$ satisfying $d(g)=c(h)$.
We will often view a group as a one-object
category.
\end{definition}

\begin{definition}
We follow the notation from 
\cite{lu06}, \cite{oinlun10} and \cite{oinlunMiyashita}.
Let $A$ be a ring.
A $G$-\emph{filter} on $A$ is a set of additive subgroups,
$\{A_g\}_{g \in \mor(G)}$, of $A$ such that for all 
$g ,h \in \mor(G)$, we have
$A_g A_h \subseteq A_{gh}$, if $(g,h) \in G^{(2)}$, 
and $A_g A_h = \{ 0 \}$ otherwise.
The ring $A$ is called $G$-\emph{graded}
if there is a $G$-filter, $\{A_g\}_{g \in \mor(G)}$ on $A$
such that $A = \oplus_{g \in \mor(G)} A_g$.
Let $A$ be a ring which is graded by a category $G$.
We say that an ideal $I$ of $A$
is \emph{graded} if $I = \oplus_{g \in \mor(G)} (I \cap A_g)$.
In that case we put $I_g = I \cap A_g$, for $g \in \mor(G)$.
We say that $A$ is \emph{locally unital} if for every
$e \in \ob(G)$ the ring $A_e$ is non-zero and unital, making
every $A_g$, for $g \in \mor(G)$, a unital
$A_{c(g)}$-$A_{d(g)}$-bimodule.
Each $a\in A$ can be written as
$a = \sum_{g\in G} r_g$ in a unique way using $r_g \in A_g$, for $g\in G$, 
of which all but finitely many are zero.
For $g\in G$, we let $(a)_g$ denote $r_g$.
The $G$-gradation on $A$ is said to be \emph{right non-degenerate}
(respectively \emph{left non-degenerate}) if to each
isomorphism $g \in \mor(G)$ and each non-zero $x \in A_g$,
the set $x A_{g^{-1}}$ (respectively $A_{g^{-1}}x$) is non-zero.
If $a \in A$ and $a = \sum_{g \in \mor(G)} a_g$
with $a_g \in A_g$, for $g \in \mor(G)$,
then we let $\Supp(a)$ denote the (finite) set of all $g \in \mor(G)$
with $a_g \neq 0$.
For a subset $H$ of $\mor(G)$, let
$A_H$ denote the direct sum $\oplus_{h \in H} A_h$.
Let $A_0$ denote $A_{\ob(G)}$.
\end{definition}

\begin{proposition}\label{Rsimpleimplies}
If $A$ is a simple ring graded by a 
category $G$ such that $A/A_0$ is ideal associative and
$A$ is locally unital, 
then $A_0$ is $A$-simple.
\end{proposition}

\begin{proof}
This follows immediately from Proposition \ref{firstprop}.
In fact, if we put $B = A_0$ and
$C = \oplus_{g \in \mor(G) \setminus \ob(G)} A_g$,
then $A = B \oplus C$ as left $B$-modules and,
since $A$ is locally unital, every
ideal of $B$ has the right identity property.
\end{proof}

\begin{proposition}\label{maximalcommutativedegreemap}
If $A$ is a ring graded by a groupoid $G$
and the gradation on $A$ is left or right non-degenerate, 
then $A/Z(A_0)$ has a degree map. 
\end{proposition}

\begin{proof}
Suppose that the gradation on $A$ is right non-degenerate.
Since the claim holds trivially if $A_0 = \{ 0 \}$,
we can assume that $A_0 \neq \{ 0 \}$.
Put $B = Z(A_0)$.
For $a \in A$, define $d(a)$ to be the cardinality of $\Supp(a)$.
Condition (d1) holds trivially.
Now we show condition (d2).
Take a non-zero ideal $I$ of $A$ and
a non-zero $a \in I$.
Take $b \in B$.
We now consider two cases.
Case 1: there is $e \in \ob(G)$
with $a_e \neq 0$. Put $a'=a$.
By the definition of the gradation on $A$,
we get that $(a'b-ba')_e = (ab-ba)_e =
a_e b - b a_e = 0$. Therefore $d(a'b-ba') < d(a)$. 
Case 2: $a_e = 0$ for all $e \in \ob(G)$.
Take a non-identity $g \in G$ with $a_g \neq 0$.
By right non-degeneracy of the gradation, there is
$c \in A_{g^{-1}}$ with $(ac)_{c(g)} \neq 0$.
Since $d(ac) \leq d(a)$ we can put $a' = ac$ and use the
argument from Case 1. The left non-degeneracy 
case is treated analogously.
\end{proof}

\begin{definition}
If $A$ is a ring which is graded by 
a category $G$, then we
say that $A/A_0$ is {\it graded ideal associative}
if every ideal of $A_0$ associates
with any combination of graded components of $A$.
It is clear that graded ideal associativity of $A/A_0$
implies ideal associativity of $A/A_0$.
At present, it is not clear to the authors
whether or not the converse holds.
\end{definition}

\begin{proposition}\label{Ainvariantequivalent}
Suppose that $A$ is a ring graded by a category $G$
such that $A$ is locally unital.
If $I$ is an ideal of $A_0$, then

\begin{itemize}

\item[{\rm (a)}] $I$ is $A$-invariant if and only if
$A_g I_{d(g)} \subseteq I_{c(g)} A_g$,
for all $g \in \mor(G)$;

\item[{\rm (b)}] if $A$ is strongly graded
by a groupoid $G$
and $A/A_0$ is graded ideal associative,
then $I$ is $A$-invariant if and only if
$A_g I_{d(g)} A_{g^{-1}} \subseteq I_{c(g)}$,
for all $g \in \mor(G)$.
\end{itemize}
\end{proposition}

\begin{proof}
(a) $I$ is $A$-invariant if and only if
$AI \subseteq IA$. Since $A$ is graded,
this inclusion holds if and only if 
$A_g I_{d(g)} \subseteq I_{c(g)} A_g$
for all $g \in \mor(G)$.

(b) We use (a).
Take $g \in \mor(G)$.
First we show the ''only if'' statement. 
Suppose that $A_g I_{d(g)} \subseteq I_{c(g)} A_g$.
Then $A_g I_{d(g)} A_{g^{-1}} \subseteq I_{c(g)} A_g A_{g^{-1}}
= I_{c(g)} A_{c(g)} \subseteq I_{c(g)}$.
Now we show the ''if'' statement.
Suppose that $A_g I_{d(g)} A_{g^{-1}} \subseteq I_{c(g)}$.
Then $A_g I_{d(g)} \subseteq A_g I_{d(g)} A_{d(g)} = A_g I_{d(g)} A_{g^{-1}} A_g
\subseteq I_{c(g)} A_g$.
\end{proof}

\begin{proposition}\label{abeliandegreemap}
If $A$ is a unital 
ring which is
strongly graded 
by an abelian group $G$
such that $A/A_e$ is graded ideal associative,
$A_e$ is $A$-simple,
then $A/A$ has a degree map. 
\end{proposition}

\begin{proof}
Put $B:=A$.
For $a \in A$, define $d(a)$ to be the cardinality of $\Supp(a)$.
Let $X = \cup_{g \in G} A_g$, i.e.
we let $X$ be the set of homogeneous elements of $A$.
It is clear that $X$ generates $A$ as a ring.
Condition (d1) holds trivially.
Now we show condition (d2).
Take a non-zero ideal $I$ of $A$ and
a non-zero $a \in I$ and $b \in X$.
We now consider two cases.

Case 1: $a_e \neq 0$. 
Let $\langle a \rangle$ denote the ideal in $A$
generated by $a$.
Now put
\begin{align*}
J = \{ c \in \langle a \rangle \mid
\Supp(c) \subseteq \Supp(a) \}_e.
\end{align*}
It is clear that $J$ is a non-zero ideal of $A_e$.
Now we show that $J$ is $A$-invariant.
By Proposition \ref{Ainvariantequivalent}(b) it is
enough to show that $A_g J A_{g^{-1}} \subseteq J$
for all $g \in G$.
Take $g \in G$ and $c \in \langle a \rangle$
with $\Supp(c) \subseteq \Supp(a)$.
Then
$\langle a \rangle \supseteq A_g c A_{g^{-1}} = 
\sum_{h \in G}
A_g c_h A_{g^{-1}}.$
Since $G$ is abelian, we get that
$A_g c_h A_{g^{-1}} \subseteq
A_{gh}A_{g^{-1}} + A_g A_{h g^{-1}} \subseteq 
A_{g h g^{-1}}
= A_h$
for any $h \in G$.
Hence $\Supp(A_g c A_{g^{-1}} ) \subseteq \Supp(a)$
and therefore $A_g c_e A_{g^{-1}} \subseteq J$.
Hence, we get that $A_g J A_{g^{-1}} \subseteq J$.
Since $A_e$ is $A$-simple we get that $J = A_e$
and hence there is $a' \in \langle a \rangle$
with $a_e' = 1_{A}$. 
Since $a_e' = 1_{A}$, $G$ is abelian
and $b$ is homogeneous, we get that
$d(a' b - b a') \leq d(a') - 1 < d(a') \leq d(a)$. 

Case 2: $a_e = 0$.
Take a non-identity $g \in G$ with $a_g \neq 0$.
By right non-degeneracy of the strong gradation, there is
$c \in A_{g^{-1}}$ with $(ac)_e \neq 0$.
Since $d(ac) \leq d(a)$ we can
use the element $ac$ and proceed as in Case 1.
\end{proof}

\begin{lemma}\label{identitylemma}
Suppose that $A$ is a locally unital
ring graded by a category $G$. 
If $I$ is an ideal of $A$, then
\begin{itemize}

\item[{\rm (a)}] $I = A$ if and only if
$1_{A_e} \in I$ for all $e \in \ob(G)$;

\item[{\rm (b)}] if $G$ is a connected groupoid
and $A$ is strongly graded,
then $I = A$ if and only if 
there is $e \in \ob(G)$ such that
$1_{A_e} \in I$.

\end{itemize}
\end{lemma}

\begin{proof}
The ''only if'' statements are trivial.
Therefore we only need to show the ''if'' statements.
It is enough to show that $I \supseteq A_g$
for all $g \in \mor(G)$.
Take $g \in \mor(G)$.
(a) Since $1_{A_{d(g)}} \in I$,
we get that $A_g = A_g 1_{A_{d(g)}} \subseteq I$.

(b) Suppose that there is $e \in \ob(G)$
such that $1_{A_e} \in I$.
Take $f \in \ob(G)$.
By (a) we are done if we can 
show that $1_{A_f} \in I$.
Since $G$ is connected, 
there is $g \in \mor(G)$ such that
$d(g)=e$ and $c(g)=f$.
Hence, by the strong gradation, we get that
$I \supseteq A_g 1_{A_e} A_{g^{-1}} = 
A_g A_{g^{-1}} = A_f \ni 1_{A_f}$. 
\end{proof}

\begin{proposition}\label{localsimplicity}
If $A$ is a ring which is locally unital and
strongly graded by a connected groupoid $G$ 
with $A_{G_e}$
simple for all $e \in \ob(G)$,
then $A$ is simple.
\end{proposition}

\begin{proof}
Take a non-zero ideal $I$ of $A$
and take a non-zero $a \in I$.
By local unitality the strong gradation is both left and right non-degenerate.
Hence we can always choose $a$ such that $a_e\neq 0$, for some $e \in \ob(G)$.
But then $J = 1_{A_e} I 1_{A_e}$
is a non-zero ideal of $A_{G_e}$.
Since $A_{G_e}$ is simple, we get that $J = A_{G_e}$.
This implies that $1_{A_e} \in J \subseteq I$.
The claim now follows from 
Lemma \ref{identitylemma}(b).
\end{proof}

\begin{proposition}\label{propgroupoidsimple}
Suppose that $A$ is a locally unital groupoid graded ring
with $A_0$ $A$-simple.
{\rm (a)} If the gradation is left or right non-degenerate,
$A_0$ is a maximal commutative subring of $A$
and every intersection of $A_0$ with an ideal of $A$
is an $A$-invariant ideal of $A_0$, then $A$ is simple;
{\rm (b)} If $A$ is strongly graded by a locally abelian
connected groupoid such that
$A/A_0$ is graded ideal associative
and for each $e \in \ob(G)$ the ring 
$Z(A_{G_e})$ is simple,
then $A$ is simple.
\end{proposition}

\begin{proof}
(a) This follows immediately from
Theorem \ref{secondtheorem} and
Proposition \ref{maximalcommutativedegreemap}.

(b) Take $e \in \ob(G)$.
We claim that the $A$-simplicity of $A_0$ implies that $A_e$ is $A_{G_e}$-simple.
If we assume that the claim holds, then it follows by Proposition \ref{abeliandegreemap} and Theorem \ref{secondtheorem},
that $A_{G_e}$ is simple.
Hence, by Proposition \ref{localsimplicity}, we get that
$A$ is simple.
Now we show the claim.
Suppose that $I_e$ is a non-zero $A_{G_e}$-invariant ideal of $A_e$.
Define the set $I$ as the sum of the additive groups $A_h I_e A_{h^{-1}} $,
for $h \in \mor(G)$ with $d(h)=e$.
Clearly, $I$ is an ideal of $A_0$.
By Proposition \ref{Ainvariantequivalent}(b) it is clear that $I$ is $A$-invariant.
Since $I_e$ is non-zero it follows that $I$ is non-zero,
and hence we get that $I = A_0$. Thus, $I \cap A_e = A_e$.
But $I \cap A_e$ equals the sum of the sets $A_h I_e A_{h^{-1}} $,
for $h \in G_e$.
Since $I_e$ is $A_{G_e}$-invariant we thus get, by Proposition \ref{Ainvariantequivalent}(b), that
$A_e = I \cap A_e \subseteq I_e$. This shows that $A_e$ is $A_{G_e}$-simple.
\end{proof}

\subsection{Category Crossed Product Algebras}\label{crossedproductalgebras}

In this section, we use the results from the
previous section to obtain both necessary
(see Proposition \ref{crossedproductsimpleimplies})
and sufficient (see Proposition \ref{crossedproductmaximalcommutativegroupoidgraded})
conditions for category crossed products to be simple.
Note that the definition that is being used here, differs slightly from the one in \cite{oinlun08}.

\begin{definition}
Let $\Ring$ denote the category of unital rings
and ring homomorphisms that respect multiplicative
identity elements.
By a \emph{crossed system} we mean a
pair $(\sigma,\alpha)$ where $\sigma$ is a 
functor $G \rightarrow \Ring$ and
$\alpha$ is a map from $G^{(2)}$
to the disjoint union of the sets $\sigma_e$,
for $e \in \ob(G)$, satisfying 
$\alpha_{g,h} 
\in \sigma_{d(g)}$, for 
$(g,h) \in G^{(2)}$.
By abuse of notation $\sigma_e$, for $e\in \ob(G)$,
is the image of the object $e$, or the image of its associated identity morphism, under the functor $\sigma$.
Hence, $\sigma_e$ will denote either an object (ring) in Ring or a morphism (ring morphism) in Ring.
It will be clear from the context, how to interpret $\sigma_e$.
In this article, we suppose that each $\alpha_{g,h}$,
for $g,h \in \mor(G)$ with $d(g)=c(h)$,
is a unit in $B_{c(g)}$ and
{\it associates and commutes with elements in $B_{c(g)}$}, 
i.e. such that the equalities 
\begin{align*}
\alpha_{g,h} (bc) = (b \alpha_{g,h}) c = 
b(\alpha_{g,h} c) = (bc) \alpha_{g,h}
\end{align*}
hold for all $b,c \in B_{c(g)}$.
We will use the notation $B_e = \sigma_e$, for 
$e \in \ob(G)$, and $B = \oplus_{e \in \ob(G)} B_e$.
Let $A = B \rtimes_{\alpha}^{\sigma} G$ denote the collection
of formal sums $\sum_{g \in G} b_g u_g$, where $b_g \in B_{c(g)}$,
$g \in \mor(G)$, are chosen so that all but finitely many of them are
zero. Define addition on $A$ pointwise
\begin{align*}
\sum_{g \in G} a_g u_g +
\sum_{g \in G} b_g u_g = 
\sum_{g \in G} \left( a_g + b_g \right) u_g
\end{align*}
and let the multiplication be defined by
the bilinear extension of the relation
\begin{equation}\label{crossedproductmultiplication}
( a_g u_g ) (b_h u_h ) =
a_g \sigma_g (b_h) \alpha_{g,h} u_{gh}
\end{equation}
for $g,h \in \mor(G)$ with $d(g)=c(h)$
and $a_g u_g b_h u_h = 0$ when $d(g) \neq c(h)$.
We call $A$ the {\it crossed product algebra}
defined by the crossed system $(\sigma,\alpha)$. 
We also suppose that $\alpha_{c(g),g} =
\alpha_{g,d(g)} = 1_{A_{c(g)}}$ for all $g \in \mor(G)$.
If we put $A_g = B_{c(g)} u_g$, for $g \in \mor(G)$,
then it is clear that this defines a gradation
on the ring $A$ which makes it locally unital
with $1_{A_e} = 1_{B_e} u_e$, for $e \in \ob(G)$.
We often identify $B$ with $A_0$.
Given a crossed product $A = B \rtimes_{\alpha}^{\sigma} G$,
we say that an ideal $I$ of $B$ is {\it $G$-invariant} if 
$\sigma_g(I_{d(g)}) \subseteq I_{c(g)}$,
for all $g \in \mor(G)$.
We say that $B$ is $G$-simple if there is no 
non-trivial $G$-invariant ideal of $B$.
\end{definition}

\begin{proposition}\label{crossedproductGinvariant}
If $A = B \rtimes_{\alpha}^{\sigma} G$ is 
a crossed product,
then an ideal of $B$ is $G$-invariant 
if and only if it is $A$-invariant.
\end{proposition}

\begin{proof}
Let $I$ be an ideal of $B$.

First we show the ''if'' statement.
Take $g \in \mor(G)$.
Suppose that $I$ is $A$-invariant.
By Proposition \ref{Ainvariantequivalent}(a),
we get that $A_g I_{d(g)} \subseteq I_{c(g)} A_g$.
By the definition of the product in $A$,
we therefore get that 
$B_{c(g)} \sigma_g( I_{d(g)} ) \subseteq I_{c(g)} B_{c(g)}$.
Since $B_{c(g)}$ is unital and $I_{c(g)}$
is an ideal of $B_{c(g)}$, we get that
$\sigma_g( I_{d(g)} ) \subseteq 
B_{c(g)} \sigma_g(I_{d(g)}) \subseteq
I_{c(g)} B_{c(g)} \subseteq I_{c(g)}$.
Hence $I$ is $G$-invariant.

Now we show the ''only if'' statement.
Suppose that $I$ is $G$-invariant.
Take
$g \in \mor(G)$. Then,
since $I_{c(g)}$ is an ideal of $B_{c(g)}$
and $B_{c(g)}$ is unital, we get that
$A_g I_{d(g)} = ( B_{c(g)} u_g ) I_{d(g)} =
B_{c(g)} \sigma_g( I_{d(g)} ) u_g \subseteq
B_{c(g)} I_{c(g)} u_g \subseteq I_{c(g)} u_g \subseteq
I_{c(g)} B_{c(g)} u_g = I_{c(g)} A_g$.
By Proposition \ref{Ainvariantequivalent}(a),
we get that $I$ is $A$-invariant.
\end{proof}

\begin{proposition}\label{crossedproductidealassociative}
If $A = B \rtimes_{\alpha}^{\sigma} G$ 
is a crossed product with every $\sigma_g$, for $g \in \mor(G)$,
surjective, then $A/B$ is graded ideal associative.
\end{proposition}

\begin{proof}
Let $I$ be an ideal of $B$.
Since each $\sigma_g$, for $g \in \mor(G)$,
is surjective, we get that images
of ideals by these maps are again ideals.
Also, since each $\alpha_{g,h}$, for $(g,h) \in G^{(2)}$,
is a unit in $B_{c(g)}$, the product
of $\alpha_{g,h}$ with an ideal of $B_{c(g)}$
equals the ideal.
Therefore, if $m$ and $n$ are non-negative
integers and $g_i \in \mor(G)$,
for $i \in \{1,\ldots,(m+n)\}$,
satisfy $d(g_i)=c(g_{i+1})$ for 
$i \in \{1,\ldots,(m+n-1)\}$,
then a straightforward induction over $m$ and $n$
shows that all the products
$A_{g_1} \cdots A_{g_m} I A_{g_{m+1}} \cdots A_{g_{m+n}}$,
performed in any prescribed order by inserting
parentheses between the factors, equals
$\sigma_{g_1 \cdots g_m} 
( I_{c(g_{m+1})} ) u_{g_{1} \cdots g_{m+n}}$.
\end{proof}

\begin{proposition}\label{crossedproductsimpleimplies}
If $G$ is a category and $A = B \rtimes_{\alpha}^{\sigma} G$ 
is a simple crossed product with every $\sigma_g$, for $g \in \mor(G)$,
surjective, then
$B$ is $G$-simple.
\end{proposition}

\begin{proof}
This follows immediately from 
Proposition \ref{Rsimpleimplies},
Proposition \ref{crossedproductGinvariant} and
Proposition \ref{crossedproductidealassociative}.
\end{proof}

\begin{proposition}\label{crossedproductmaximalcommutativegroupoidgraded}
Suppose that $G$ is a groupoid and
$A = B \rtimes_{\alpha}^{\sigma} G$ 
is a crossed product where $B$ is a $G$-simple ring.
{\rm (a)} If $B$ is a maximal commutative subring of $A$, then $A$ is simple;
{\rm (b)} If $G$ is locally abelian and connected, and
for each $e \in \ob(G)$, the ring $Z(A_{G_e})$
is simple, then $A$ is simple. 
\end{proposition}

\begin{proof}
First note that since $G$ is a groupoid,
every $\sigma_g$, for $g \in \mor(G)$,
is an isomorphism and hence, in 
particular, a surjection.
Indeed, 
$\sigma_g \sigma_{g^{-1}} = 
\sigma_{c(g)} = \id_{B_{c(g)}}$
and $\sigma_{g^{-1}} \sigma_g = 
\sigma_{d(g)} = \id_{B_{d(g)}}$.

(a) Let $I$ be an ideal of $A$.
By Proposition \ref{propgroupoidsimple} and
Proposition \ref{crossedproductGinvariant}
we are done if we can show that $I \cap B$
is $G$-invariant. If $g \in \mor(G)$,
then $\sigma_g (I_{d(g)}) =
\sigma_g( I_{d(g)} ) 
\alpha_{g,g^{-1}} \alpha_{g,g^{-1}}^{-1} =
u_g I_{d(g)} u_{g^{-1}} \alpha_{g,g^{-1}}^{-1} 
\subseteq I \cap B_{c(g)}  = I_{c(g)}$.

(b) This follows immediately from
Proposition \ref{propgroupoidsimple},
Proposition \ref{crossedproductGinvariant} and
Proposition \ref{crossedproductidealassociative}.
\end{proof}

\begin{remark}\label{twodifferentproducts}
In the definition of crossed product algebras
one may loosen the definition of the product
\eqref{crossedproductmultiplication} in the 
following way. For each
$(g,h) \in G^{(2)}$ we let $*_{g,h}$
denote either the ordinary multiplication
on $B_{c(g)}$ or the opposite multiplication on $B_{c(g)}$.
Then we define
$(au_g)(b u_h) = a *_{g,h} \sigma_g(b) \alpha_{g,h} u_{gh},$
for $(g,h) \in G^{(2)}$, $a \in B_{c(g)}$ and
$b \in B_{c(h)}$. 
One may also start with a functor $\sigma$ from $G$
to the larger category $\Ring'$ with unital rings as objects 
and as morphisms functions that are
ring homomorphisms or ring anti-homomorphisms.
It is straightforward to verify that 
Proposition \ref{crossedproductsimpleimplies} and
Proposition \ref{crossedproductmaximalcommutativegroupoidgraded}
also hold in these cases.
\end{remark}

\begin{example}
There are group graded rings that are not
in a natural way crossed products by groups.
This serves as motivation for the results in Section
\ref{categorygradedrings}.
To exemplify this, let $B$ be an associative,
commutative and unital ring.
Then the ring $A = M_3(B)$ can 
be strongly graded by ${\Bbb Z}_2 = \{ 0 , 1 \}$
by putting
\[ 
A_0 = 
\left( 
\begin{array}{ccc}
B & B & 0 \\
B & B & 0 \\
0 & 0 & B 
\end{array} 
\right) 
\quad
\mbox{and}
\quad
A_1 =
\left( 
\begin{array}{ccc}
0 & 0 & B \\
0 & 0 & B \\
B & B & 0 
\end{array} 
\right). 
\] 
With respect to this gradation,
$A$ is not a crossed product.
Indeed, if there were $u$ and $v$ in $A_1$
satisfying $uv = 1_A$, then we would get
a contradiction, 
since the determinant of every
element in $A_1$ is zero.
Furthermore, it is clear that every ideal 
of $A_0$ is of the form 
\[ 
I_{J,K} = 
\left( 
\begin{array}{ccc}
J & J & 0 \\
J & J & 0 \\
0 & 0 & K 
\end{array} 
\right) 
\]
for ideals $J$ and $K$ of $B$.
By Proposition \ref{Ainvariantequivalent}(b)
we get that $I_{J,K}$ is $A$-invariant
if and only if $A_1 I_{J,K} A_1 \subseteq I_{J,K}$.
It is easy to check that
this holds precisely when $J=K$.
So $A_0$ is $A$-simple if and only if 
$B$ is simple.
Therefore, by Proposition \ref{Rsimpleimplies} 
and Proposition \ref{propgroupoidsimple},
we get that $A$ is simple precisely when $A_0$
is $A$-simple.
\end{example}

\subsection{Cayley-Dickson Doublings}\label{applicationcayley}

In this section, we analyze simplicity of
Cayley-Dickson doublings (see Proposition \ref{simplecayley}).
We also show how to construct an infinite
chain of simple rings using this procedure
(see Example \ref{examplecayley}).

\begin{definition}
By a {\it Cayley-Dickson doubling} we mean
a crossed product algebra $A = B \rtimes_{\alpha}^{\sigma} G$,
where $G$ is a group of order two and $B$ is a unital ring.
We let the identity element of $G$ be denoted 
by $e$ and suppose that $g$ is a generator for $G$.
Suppose that $\sigma$ is a functor from $G$
(considered as a one-object category) to $\Ring'$.
We let $\sigma_g$ be denoted by $\sigma$
and we let $B_e$ be denoted by $B$.
Then $\sigma : B \rightarrow B$ is
either a ring automorphism or a ring
anti-automorphism satisfying $\sigma^2 = \id_B$.
We let $\alpha$ denote $\alpha_{g,g}$
and we suppose that $\alpha$ is a unit in $B$. 
We let $1$ denote $1_B u_e$.
By Remark \ref{twodifferentproducts},
the product in $A$ is given by
\begin{align*}
(a + b u_g) (c + d u_g) = 
a *_{e,e} c + b *_{g,g} \sigma(d) \alpha + 
(a *_{e,g} d + b *_{g,e} \sigma(c))u_g
\end{align*}
for $a,b,c,d \in B$.
An ideal $I$ of $B$ is called
$\sigma$-invariant if $\sigma(I) \subseteq I$.
The ring $B$ is called $\sigma$-simple
if there is no non-trivial $\sigma$-invariant
ideal of $B$.
By a {\it classical Cayley-Dickson doubling} we mean
a doubling where
$\sigma$ is a ring anti-automorphism of $B$ and
$a *_{e,e} b = ab$, $a *_{e,g} b = ba$,
$a *_{g,e} b = ab$ and $a *_{g,g} b = ba$
for all $a,b \in B$.
Recall that in this case $\sigma$
can be extended to a 
to a ring 
anti-automorphism of $A$ by the relation
$\sigma (a + b u_g) = \sigma(a) - bu_g$
for all $a,b \in B$. 
\end{definition}

\begin{proposition}\label{simplecayley}
Suppose that $A$ is a (classical)
Cayley-Dickson doubling of $B$.
{\rm (a)} If $A$ is simple ($\sigma$-simple), then $B$ is $\sigma$-simple.
{\rm (b)} If $B$ is $\sigma$-simple and $Z(A)$
is simple, then $A$ is simple.
\end{proposition}

\begin{proof}
(a) follows from Proposition \ref{crossedproductsimpleimplies}.
Indeed, if $I$ is a $\sigma$-invariant
ideal of $B$, then $IA$ is also a $\sigma$-invariant
ideal of $A$, since
$\sigma(IA) =
\sigma(I + Iu_g) = \sigma(I) - Iu_g = I + Iu_g = IA$.
%= i(I)$.
(b) follows from Proposition 
\ref{crossedproductmaximalcommutativegroupoidgraded}.
\end{proof}

\begin{example}\label{examplecayley}
Now, suppose that we are given 
an infinite tower of ring extensions
$B_0 \subseteq B_1 \subseteq \ldots$
and units $\alpha_i \in B_i$, for $i \geq 0$,
such that $\alpha_i$ commutes and associates
with any two elements of $B_i$.
Furthermore, suppose that we for each
$i \geq 0$ are given
a ring anti-automorphism $\sigma_i$ of $B_i$
and that the rings $B_0, B_1, B_2, \ldots$ are defined
recursively by the relations
$B_0 = B$ and $B_{i+1}$ equals
the classical Cayley-Dickson doubling of $B_i$
defined by $\sigma_i$ and $\alpha_i$.
By Proposition \ref{simplecayley}(b) we get:
\begin{itemize}
\item If $B_0$ is $\sigma_0$-simple and 
$Z(B_i)$ is simple for each $i \geq 1$,
then $B_i$ is simple for each $i \geq 1$.
\end{itemize}
Starting with the real numbers $B_0 = {\Bbb R}$
equipped with the trivial action of $\sigma$
and putting $\alpha_i = -1$, for $i \geq 0$,
we get successively, $B_1$ the complex numbers
${\Bbb C}$, $B_2$ Hamilton's quaternions ${\Bbb H}$,
$B_3$ the octonions ${\Bbb O}$, $B_4$ the sedenions ${\Bbb S}$
and so on. 
It is easy to check that the
center of each of these rings, except for $B_1={\Bbb C}$, equals
the real numbers and is hence a simple
ring. We therefore, as a special case,
get the following well-known result.
\begin{itemize}
\item $B_1 = {\Bbb C}$ is simple and all of the rings 
$B_0 = {\Bbb R}$, $B_2 = {\Bbb H}$, $B_3 = {\Bbb O}$, $B_4 = {\Bbb S}$, 
etc. are central simple ${\Bbb R}$-algebras.
\end{itemize}
The question of when a twisted Cayley-Dickson doubling of a 
quaternion algebra is a division algebra
has been studied before (see e.g. \cite{pumplun}).
This of course implies simplicity, 
but it seems to the authors of the present article that the above result concerning
simplicity of the whole chain of 
Cayley-Dickson doublings is new.
\end{example}

\subsection{Twisted Group Rings}

In this section, we obtain a simplicity 
result (see Proposition \ref{simpletwisted}) for
twisted group rings.
This provides us with a different proof
of simplicity of the Cayley-Dickson doubling
algebras from Example \ref{examplecayley}.

\begin{definition}
Let $G$ be a group with identity element $e$.
By a twisted group ring we mean a crossed product 
$A = B \rtimes_{\alpha}^{\sigma} G$ where
$B$ is a unital, not necessarily associative,
ring, $\sigma_g = {\rm id}_B$, for $g \in G$,
and $\alpha$ is a map from $G \times G$ to 
the units of $B$ such that 
$\alpha_{e,g}=\alpha_{g,e}=1_B$, for $g \in G$,
and each $\alpha_{g,h}$, for $g,h \in G$,
associates and commutes with all elements of $B$.
In that case we write $A = B \rtimes_{\alpha} G$.
\end{definition}

\begin{proposition}\label{simpletwisted}
Let $A = B \rtimes_{\alpha} G$ be a twisted 
group ring, where $G$ is abelian, and
$B$ and $Z(B)$ are simple.
If for each non-identity $g \in G$,
there is a non-identity $h \in G$ such that
$\alpha_{g,h} - \alpha_{h,g}$ is a unit in $B$,
then $A$ is simple.
\end{proposition}

\begin{proof}
By Proposition \ref{crossedproductmaximalcommutativegroupoidgraded}(b)
it is enough to show that $Z(A) = Z(B)$.
The inclusion $Z(A) \supseteq Z(B)$ is trivial.
Now we show the inclusion $Z(A) \subseteq Z(B)$.
To this end, suppose that 
$x = \sum_{p \in G} b_p u_p$ belongs to $Z(A)$,
where $b_p \in B$ is zero
for all but finitely many $p \in G$.
Fix a non-identity $g \in G$.
Choose a non-identity $h \in G$
such that $\alpha_{g,h} - \alpha_{h,g}$ is a unit in $B$.
From the relation $x u_h = u_h x$ and
the fact that $G$ is abelian, we get that
$b_g \alpha_{g,h} = b_g \alpha_{h,g}$,
which, since $\alpha_{g,h}-\alpha_{h,g}$
is a unit in $B$, implies that $b_g=0$.
Therefore $x \in B$.
But since $B \subseteq A$ and $x \in Z(A)$, 
we get that $x \in Z(B)$.
\end{proof}

\begin{example}
Now we can use Proposition \ref{simpletwisted}
to deduce simplicity of the algebras 
${\Bbb R}$, ${\Bbb C}$, ${\Bbb H}$, 
${\Bbb O}$, ${\Bbb S}$ etc. (see Example \ref{examplecayley})
in a different way.
In fact, Bales \cite{bales} has shown that
the $n$th Cayley-Dickson doubling of ${\Bbb R}$
is a twisted group algebra over the group ${\Bbb Z}_2^{2^n}$
with $\alpha : {\Bbb Z}_2^{2^n} \times {\Bbb Z}_2^{2^n} \rightarrow \{ 1 , -1 \}$
defined recursively by the following five relations
\begin{align*}
\alpha(p,0)=\alpha(0,p)=1, \quad
 \alpha(2p,2q)=\alpha(p,q), \quad
\alpha(1,2q+1)=-1, \\
 \alpha(2p,2q+1) = -\alpha(p,q), \quad \mbox{and} \quad
\alpha(2p+1,2q+1)= \alpha(q,p), \quad \mbox{for $p \neq 0$,}
\end{align*}
for integers $p$ and $q$.
In particular, this implies that 
$\alpha$ is anti-commutative,
in the sense that $\alpha_{p,q} = -\alpha_{q,p}$ holds
for all non-zero $p,q \in {\Bbb Z}_2^{2^n}$.
Hence, for all non-zero $p,q \in {\Bbb Z}_2^{2^n}$,
we get that $\alpha_{p,q} - \alpha_{q,p}$ equals
either $2$ or $-2$, which, in either case,
is a real unit.
\end{example}

\subsection{Matrix Algebras}\label{applicationmatrix}

In this section, we obtain a
simplicity result for skew and
twisted matrix algebras 
(see Proposition \ref{simplematrix}).

\begin{definition}\label{defmatrix}
Suppose that $I$ is a (possibly infinite) set and let
$G$ be the groupoid with $\ob(G) = I$
and $\mor(G) = I \times I$.
For $i,j \in I$, we put
$d(i,j)=j$ and $c(i,j)=i$ so we can write
$(i,j) : j \rightarrow i$.
If $i,j,k \in I$, then 
 $(i,j)(j,k)$ is defined to be $(i,k)$.
Suppose that we are given unital rings 
$B_i$, for $i \in I$, and ring isomorphisms
$\sigma_{ij} : B_j \rightarrow B_i$,
for $i,j \in I$, such that 
$\sigma_{ii} = \id_{B_i}$, for $i \in I$,
and $\sigma_{ij}\sigma_{jk} = \sigma_{ik}$
for all $i,j,k \in I$.
Suppose also that we, for each triple 
$(i,j,k) \in I \times I \times I$, are given
a unit $\alpha_{ijk} \in B_i$,
with $\alpha_{ijk}=1_{B_i}$ 
whenever $i=j$ or $j=k$, 
that commute and associate with any two elements of $B_i$.
%Suppose also that
For such a triple,
we let $*_{ijk}$
denote either the ordinary multiplication
on $B_i$ or the opposite multiplication on $B_i$.
Then the crossed product 
$A = B \rtimes_{\alpha}^{\sigma} G$
can be interpreted as the set of matrices,
over the index set $I$ with entries in
the rings $B_i$, with a product which is
both ''skew'' and ''twisted''
by the $\alpha_{ijk}$'s, the $*_{ijk}$'s
and the $\sigma_{ij}$'s and we write
$A = M_I(B)$.
\end{definition}

\begin{proposition}\label{simplematrix}
With the above notation, the ring $M_I(B)$
is simple if and only if for each $i \in I$,
the ring $B_i$ is simple.
In particular, all of the rings
$M_I({\Bbb R})$, 
$M_I({\Bbb C})$,
$M_I({\Bbb H})$,
$M_I({\Bbb O})$,
$M_I({\Bbb S})$ etc.
are simple.
\end{proposition}

\begin{proof}
Note that since each each $G_e$, for $e \in \ob(G)$,
is a trivial group, we get that $B$ is $G$-simple 
precisely when each $B_i$, for $i \in I$, is simple.
Therefore, the first part of the result follows from
Proposition \ref{localsimplicity} and Proposition \ref{crossedproductsimpleimplies}.
The second part of the result follows
from the first part and Example \ref{examplecayley}.
\end{proof}

\subsection{Topological Dynamics}\label{applicationtopological}

In this section, we specialize the simplicity
results for crossed products obtained
earlier to show a simplicity
result for skew group rings, over non-associtive rings,
associated with topological dynamical systems
(see Proposition \ref{simplegroupdynamics}).

\begin{definition}
Let $G$ be a group with identity element $e$.
By a \emph{skew group ring} we mean a crossed product 
$A = B \rtimes_{\alpha}^{\sigma} G$ where
$B$ is a unital, not necessarily associative,
ring, $\sigma_g$ is a ring automorphism of $B$, for $g \in G$,
and $\alpha_{g,h}=1_B$, for $g,h \in G$.
In that case we write $A = B \rtimes^{\sigma} G$.
\end{definition}

\begin{definition}
Let $G$ be a group and
suppose that $X$ is a compact Hausdorff
space.
Let $\Aut(X)$ denote the group of
homeomorphisms of $X$. 
By a \emph{topological dynamical system}
we mean a group homomorphism $s : G \rightarrow \Aut(X)$.
A subset $Y$ of $X$ is called \emph{$G$-invariant} if
$s(g)(Y) \subseteq Y$ holds for all $g \in G$. 
The topological dynamical system $s$
is said to be \emph{minimal} if there is no
closed non-empty proper invariant subset of $X$.
Furthermore, $s$ is called {\it faithful}
if there to each non-identity $g \in G$, 
is some $x \in X$ such that $s(g)(x) \neq x$.
For every non-negative integer $i$,
let $C_i$ denote the  $i$th Cayley-Dickson
doubling starting with $C_0 = {\Bbb R}$
and taking $\alpha = -1$ at each step.
Then we get $C_1 = {\Bbb C}$,
$C_2 = {\Bbb H}$, $C_3 = {\Bbb O}$, $C_4 = {\Bbb S}$
and so on.
By abuse of notation we let
the norm (respectively involution) 
on each $C_i$, for $i \geq 0$,
be denoted by $| \cdot |$ (respectively $\overline{\cdot}$).
For each $i \in I$, let $B_i$ denote the algebra of 
$C_i$-valued continuous functions on $X$.
The addition and multiplication on $B_i$ are
defined pointwise and the elements of $C_i$
are identified with the constant functions
from $X$ to $C_i$. We define a norm $\| \cdot \|$
on each $B_i$, for $i \geq 0$, by 
$\| f \| = \sup_{x \in X} | f(x) |$, for $f \in B_i$.
For each $i \geq 0$, the group homomorphism $s$ 
induces a group homomorphism
$\sigma : G \rightarrow \Aut(B_i)$ by
$\sigma_g(f)(x) = f(s(g^{-1})(x))$,
for $g \in G$, $f \in B_i$ and $x \in X$.
This allows us to define
the skew group rings $B_i \rtimes^{\sigma} G$,
for $i \geq 0$.
\end{definition}

Öinert \cite{oinert09,oinertarxiv11} has shown 
that the following result holds in the case $i=1$.

\begin{proposition}\label{simplegroupdynamics}
If $G$ is abelian, then for each $i\geq 0$, the skew group ring
$B_i \rtimes^{\sigma} G$ is simple if 
and only if $s$ is minimal and faithful.
\end{proposition}

\begin{proof}
Fix $i \geq 0$ and put $A_i = B_i \rtimes^{\sigma} G$.

Suppose first that $A_i$ is simple.
By Proposition \ref{crossedproductsimpleimplies},
we get that $B_i$ is $G$-simple.
Then $s$ is minimal. In fact, suppose that
$Y$ is a non-empty closed $G$-invariant subset of $X$.
Let $I_Y$ denote the ideal of $B_i$ consisting of all $f \in B_i$
that vanish on $Y$. Since every compact
Hausdorff space is regular, it follows
that $I_Y$ is non-zero. Since $Y$ is $G$-invariant it follows that $I_Y$ is
$G$-invariant, and hence by the $G$-simplicity
of $B_i$, we get that $I_Y = B_i$. Thus,
$Y = X$. Now we show that $s$ is faithful.
It is easy to see that this is equivalent
to $\sigma$ being injective (see e.g. \cite{oinertarxiv11}).
Seeking a contradiction, suppose that
there is a non-identity morphism $g \in G$ with
$\sigma_g = \id_{B_i}$.
Let $I$ be the ideal of $A_i$
consisting of finite sums of elements
of the form $b (u_{hh'} - u_{hgh'})$
for $b \in B_i$ and $h,h' \in G$.
Then $I$ is non-zero (since $g$ is not an identity morphism)
and proper (since the sum of the coefficients
of each element in $I$ equals zero) giving us
a contradiction. Therefore $s$ is faithful.

Now suppose that $s$ is both minimal and faithful.
We show that $A_i$ is simple for $i \neq 1$
(the case $i=1$ was already treated in \cite{oinert09,oinertarxiv11}).
Again faithfulness is equivalent to
injectivity of $\sigma$.
Since $s$ is minimal we get that
$B_i$ is $G$-simple. In fact,
seeking a contradiction, suppose that there
is a non-trivial $G$-invariant ideal $I$ of $B_i$.
For a subset $J$ of $B_i$,
let $N_J$ denote the set 
$\bigcap_{f \in J} f^{-1}( \{ 0 \} )$.
We claim that $N_I$ is a closed, non-empty
proper $G$-invariant subset of $X$.
If we assume that the claim holds,
then $s$ is not minimal and we have reached a contradiction.
Now we show the claim.
Since $I$ is $G$-invariant
the same is true for $N_I$.
Since $I$ is non-zero it follows that $N_I$
is a proper subset of $X$.
Since each set $f^{-1}( \{ 0 \} )$, for $f \in I$,
is closed, the same is true for $N_I$.
Seeking another contradiction, suppose that $N_I$ is empty.
Since $X$ is compact,
there is a finite subset $J$ of $I$ such that
$N_J = N_I$. Then the function
$F = \sum_{f \in J} |f|^2 =
\sum_{f \in J} f \overline{f}$ belongs to $I$
and, since $N_J$ is empty, 
it has the property that $F(x) \neq 0$ for all $x \in X$.
Therefore $1 = F \cdot \frac{1}{F} \in I$.
This implies that $I = B_i$
which is a contradiction.
Therefore $N_I$ is non-empty.
To show that $A_i$ is simple it is,
by Proposition \ref{crossedproductmaximalcommutativegroupoidgraded}(b), 
enough to show that $Z(A_i)$ is simple.
We will in fact show that $Z(A_i)$ equals the field
${\Bbb R}$ which in particular is a simple ring.
It is clear that $Z(A_i) \supseteq {\Bbb R}$.
Now we show the reversed inclusion.
Take $\sum_{g \in G} f_g u_g \in Z(A_i)$ 
where $f_g \in B_i$, for $g \in G$,
and $f_g = 0$ for all but finitely many $g \in G$.
For every $h \in G$, the equality
$u_h \left( \sum_{g \in G} f_g u_g \right) 
= \left( \sum_{g \in G} f_g u_g \right) u_h$
holds. From the fact that $G$ is abelian,
we get that $f_g(s(h^{-1})(x)) = f_g(x)$,
for $g,h \in G$ and $x \in X$.
For every $g \in G$ choose a number
$z_g$ in the image of $f_g$.
Since $f_g \circ s(h^{-1}) = f_g$, for each $h\in G$, it follows that
the set $f_g^{-1}(z_g)$ is a non-empty
$G$-invariant closed subset of $X$.
Since $s$ is minimal it follows that 
$f_g^{-1}(z_g) = X$ and hence that 
$f_g = z_g$, for $g \in G$.
Take a non-identity $h \in G$.
From the fact that $s$ is faithful,
we get that there is $a \in X$
such that $s(h^{-1})(a) \neq a$.
Since $X$ is a compact Hausdorff space,
it follows by Urysohn's lemma that
there is a $p \in B_i$, taking its values in ${\Bbb R}$, such that
$p(a) \neq p(s(h^{-1})(a))$.
Since $\sum_{g \in G} z_g u_g$
commutes with $p$, we get that
$z_g ( p(x) - \sigma(g)(p)(x) ) = 0$
for all $x \in X$ and $g\in G$. By specializing
this equality to $x=a$ and $g=h$, we get that
$z_h (p(a) - p(s(h^{-1})(a)) = 0$
which in turn implies that $z_h = 0$.
Therefore $Z(A_i) \subseteq {\Bbb R}$. 
\end{proof}

\subsection{Ore Extensions}\label{oreextensions}

In this section, we show that
Theorem \ref{oinricsilmaximalcommutative}, 
which recently appeared in \cite{oinricsil11}, follows
from Theorem \ref{secondtheorem}. 

\begin{definition}
Let $B$ be an associative and unital ring
equipped with a ring endomorphism
$\sigma : B \rightarrow B$, and a 
$\sigma$-derivation $\delta$, 
i.e. an additive endomorphism of $B$ satisfying
$\delta(bc) = \sigma(b)\delta(c) + \delta(b)c$,
for $b,c \in B$.
The Ore extension $A$ of $B$ is defined to be polynomials
over $B$ in the indeterminate $x$,
subject to the
relation $xb = \sigma(b)x + \delta(b)$, for $b \in B$.
For a non-zero element $a \in B[x ; \sigma , \delta]$, 
we let $\deg(a)$ denote its
degree as a polynomial in $x$.
An ideal $I$ of $B$ is called $\sigma$-$\delta$-invariant
if $\sigma(I) \subseteq I$ and $\delta(I) \subseteq I$ hold.
The ring $B$ is called $\sigma$-$\delta$-simple
if there is no non-trivial $\sigma$-$\delta$-invariant
ideal of $B$.
If $\sigma = \id_B$, then $A$ is 
called a {\it differential polynomial ring} in which case 
$\sigma$-$\delta$-simplicity coincides with $\delta$-simplicity.
\end{definition}

\begin{proposition}\label{oreidealABideal}
If $A = B[x ; \sigma , \delta]$
is an Ore extension,
then
an ideal of $B$ is 
$\sigma$-$\delta$-invariant if and only if
it is $A$-invariant.
\end{proposition}

\begin{proof}
Let $I$ be an ideal of $B$.

First we show the ''if'' statement.
Suppose that $I$ is $A$-invariant.
Take $b \in I$. Then
$\sigma(b)x + \delta(b) = xb \in xI
\subseteq IA$.
Since $IB=I$ and $A$
is a free left $B$-module with 
the non-negative powers of $x$ as a basis, 
this implies that $\sigma(b) \in I$ 
and $\delta(b) \in I$.
Hence $\sigma(I) \subseteq I$ 
and $\delta(I) \subseteq I$.
Therefore $I$ is $\sigma$-$\delta$-invariant.

Now we show the ''only if'' statement.
Suppose that $I$ is $\sigma$-$\delta$-invariant.
We claim that if $n$ is a non-negative integer, 
then the inclusion $B x^n I \subseteq I A$ holds.
If we assume that the claim holds, then,
since $A$ is a free left $B$-module
with the non-negative powers of $x$ as a basis,
we get that 
$A I \subseteq I  A$
and hence that $I$ is an 
$A$-invariant ideal.
Now we show the claim by induction over $n$.
Since $B x^0 I = B 1_B I = B I = I =
IB \subseteq IA$,
the claim holds for $n=0$.
The claim also holds for $n=1$ since
$B x I \subseteq B \sigma(I) x + B \delta(I) 
\subseteq B I x + BI = Ix + I \subseteq IA$.
Suppose now that the claim holds for 
an integer $n \geq 1$.
By the induction hypothesis and
the case $n=1$, we get that 
$B x^{n+1} I = B x x^n I \subseteq Bx B x^n I \subseteq
Bx I A \subseteq I A A = I A$,
which completes the proof.
\end{proof}

The following result appeared in \cite{oinricsil11}.
We show that it follows from the results
in Section \ref{mainresults}.

\begin{proposition}\label{injectiveore}
If $A = B[x ; \sigma , \delta]$
is a simple Ore extension,
then $B$ is $\sigma$-$\delta$-simple.
\end{proposition}

\begin{proof}
This follows immediately from Theorem \ref{firsttheorem}
and Proposition \ref{oreidealABideal}.
\end{proof}

\begin{proposition}\label{degreemapA/Z(B)}
If $A = B[x ; \delta]$ is a differential polynomial ring,
then
{\rm (a)} $A/Z(B)$ has a degree map;
{\rm (b)} if $B$ is $\delta$-simple, then
$A/A$ has a degree map and thereby $A/Z(A)$ has the ideal intersection property.
\end{proposition}

\begin{proof}
For a non-zero $p \in A$, define
$d(p) = \deg(p)+1$. Let $d(0)=0$.

(a) Let $X = Z(B)$.
Condition (d1) is trivially satisfied.
Now we check condition (d2).
Take a non-zero ideal $I$ of $A$ and
a non-zero $a \in I$ and $b \in X$.
If $d(a) = 1$, then, trivially,
$ab-ba=0$ and, hence, $d(ab-ba)=d(0)=0 < d(a)$.
Now suppose that $d(a) = n+1 \geq 2$, i.e. $\deg(a) = n \geq 1$.
For each non-negative integer $i$, we may write $x^i b = \sum_{j=0}^i s_{i,i-j}(b) x^j$
using some elements $s_{i,0}(b),s_{i,1}(b),\ldots,s_{i,i}(b)$ of $B$.
Let $a = \sum_{i=0}^n b_i x^i$, for $\{b_0,b_1,\ldots,b_n\} \subseteq B$,
with $b_n \neq 0$. Then, since $b \in Z(B)$, we get that
\begin{eqnarray}
ab - ba &=
\left( \sum_{i=0}^n b_i x^i \right)b - 
b \left( \sum_{i=0}^n b_i x^i \right) =
\sum_{i=0}^n (b_i x^i b - b b_i x^i)  \nonumber \\
&=
\sum_{i=0}^n \left( \left( \sum_{j=0}^i b_i s_{i,i-j}(b) x^j \right) - b b_i x^i \right). \label{eq:ABBAcalculation}
\end{eqnarray}
Since $b \in Z(B)$, we get that the coefficient 
of the term of degree $n$ equals
$b_n s_{n,0}(b) - b b_n = b_n b - b b_n = 0$.
Therefore, we get that $d(ab-ba) \leq n < n+1 = d(a)$.
So we can put $a'=a$.
This shows (d2).

(b) Let $X = B \cup \{ x \}$.
Condition (d1) is trivially satisfied.
Now we check condition (d2).
Take a non-zero ideal $I$ of $A$ and
a non-zero $a \in I$ and $b \in X$.
By an argument similar to the graded case (Proposition \ref{abeliandegreemap}) we may assume
that the highest degree coefficient of $a$ is 1
(see Lemma 4.13 in \cite{oinricsil11}), i.e. $a = \sum_{i=0}^n b_i x^i$ where $b_n = 1$.
We now consider two cases.
Case 1: Take $b \in B$. By the calculation in \eqref{eq:ABBAcalculation},
we get that $d(ab - ba) < d(a)$.
Case 2: Let $b = x$. 
Then, since $\delta(1)=0$, we get that
\begin{align*}
ax - xa &= \left( \sum_{i=0}^n b_i x^i \right)x - 
x \left( \sum_{i=0}^n b_i x^i \right) =
\sum_{i=0}^n ( b_i x^{i+1} - (b_i x + \delta(b_i))x^i ) \\
&= \sum_{i=0}^n (b_i x^{i+1} - b_i x^{i+1} - \delta(b_i)x^i) =
-\sum_{i=0}^n \delta(b_i) x^i =
-\sum_{i=0}^{n-1} \delta(b_i) x^i.
\end{align*}
Therefore, we get that
$d(ax - xa) \leq n < n+1 = d(a)$.
So we can put $a'=a$.
This shows (d2).
The final conclusion now follows from Lemma \ref{thirdtheorem}.
\end{proof}

\noindent {\bf Proof of Theorem \ref{oinricsilmaximalcommutative}.}
(a) This follows immediately from Theorem \ref{secondtheorem}, Proposition \ref{oreidealABideal} 
and Proposition \ref{degreemapA/Z(B)}(a), if we can show
that every intersection of $B$ with an ideal of $A$
is $\delta$-invariant. Take an ideal $I$ of $A$.
If $b \in I \cap B$, then
$\delta(b) = bx + \delta(b) - bx = xb - bx \in I \cap B$.

(b) The necessary part follows immediately from
Proposition \ref{injectiveore}
and the fact that the center of any associative
unital ring is a field.
The sufficient part follows immediately 
from Theorem \ref{secondtheorem} 
and Proposition \ref{degreemapA/Z(B)}(b).
{\hfill $\qed$}

\section*{Acknowledgements}
The second author was partially supported by The Swedish Research Council (postdoctoral fellowship no. 2010-918
and repatriation grant no. 2012-6113) and
The Danish National Research Foundation (DNRF) through the Centre for Symmetry and Deformation.

\end{document}